\documentclass[twoside, 11pt]{article}
\usepackage[noadjust]{cite}
\usepackage{booktabs}
\usepackage{mathrsfs}
\usepackage{amssymb, amsmath, mathrsfs, amsthm}
\usepackage{graphicx}
\usepackage{color}
\usepackage[top=3cm, bottom=2cm, left=2cm, right=2cm]{geometry}
\usepackage{float, caption}
\usepackage{cite}
\usepackage{tikz}
\usepackage{subfigure}
\usepackage{enumerate}
\usepackage[colorlinks,
linkcolor=red,
anchorcolor=blue,
citecolor=green
]{hyperref}
\usepackage{diagbox}
\input{epsf}

\newcommand{\bd}{\begin{description}}
\newcommand{\ed}{\end{description}}
\newcommand{\bi}{\begin{itemize}}
\newcommand{\ei}{\end{itemize}}
\newcommand{\be}{\begin{enumerate}}
\newcommand{\ee}{\end{enumerate}}
\newcommand{\beq}{\begin{equation}}
\newcommand{\eeq}{\end{equation}}
\newcommand{\beqs}{\begin{eqnarray*}}
\newcommand{\eeqs}{\end{eqnarray*}}

\definecolor{DarkGreen}{rgb}{0.2, 0.6, 0.3}

\newtheorem{theorem}{Theorem}[section]

\newtheorem{definition}[theorem]{Definition}
\newtheorem{corollary}[theorem]{Corollary}
\newtheorem{observation}[theorem]{Observation}

\newtheorem{operation}[theorem]{Operation}
\newtheorem{case}{Case}

\newtheorem{remark}{Remark}[section]

\makeatletter
\newenvironment{breakablealgorithm}
{
	\begin{center}
		\refstepcounter{algorithm}
		\hrule height.8pt depth0pt \kern2pt
		\renewcommand{\caption}[2][\relax]{
			{\raggedright\textbf{\ALG@name~\thealgorithm} ##2\par}%
			\ifx\relax##1\relax 
			\addcontentsline{loa}{algorithm}{\protect\numberline{\thealgorithm}##2}%
			\else 
			\addcontentsline{loa}{algorithm}{\protect\numberline{\thealgorithm}##1}%
			\fi
			\kern2pt\hrule\kern2pt
		}
	}{
	\kern2pt\hrule\relax
\end{center}
}
\makeatother
\usepackage{algorithm}
\usepackage{algorithmic}

\begin{document}
\title{\textbf{A note on the precoloring extension of
outerplane graphs}\thanks{
 \it $^c$Corresponding author. xcdeng@mail.tjnu.edu.cn}}
\author{ Xingchao Deng$^{a,b,c}$, Beiyan Zou$^a$, Hong Zhai$^a$ \\
	\footnotesize $^{a}$  College of Mathematical Science, Tianjin  Normal University\\
	\footnotesize Tianjin , 300387, P. R. China\\
        \footnotesize $^{b}$ Institute of Mathematics and Interdisciplinary Sciences, Tianjin Normal University\\
      \footnotesize Tianjin , 300387, P. R. China }
\date{}
\maketitle
\begin{abstract}

  The \emph{precoloring extension problem} involves initially assigning colors to certain vertices and determining whether such a partial coloring can be extended to a \emph{proper k-coloring} of the entire graph for an integer $k$. In a landmark  result, Gr\"{o}tzsch demonstrated that every planar graph without triangle is 3-colorable. A significant generalization of this theorem was recently achieved by La et al. \cite{LaLuSt}. In this paper, we establish that any precoloring of three independent vertices (or two independent vertices) can be extended to a 3-coloring of the entire graph when it is a connected outerplane graph containing at most one triangle (or two triangles, respectively). \\[2mm]


\end{abstract}
\section{Introduction}
A graph $G$ is termed a plane graph if it is embedded in the plane without edge crossings. A graph is planar if and only if it does not contain
$K_5$  or $K_{3,3}$ minor, according to Kuratowski's theorem. A Hamilton cycle is defined as a cycle in a graph that passes through every vertex exactly once.
A graph is outerplanar if it can be embedded in the plane such that all its vertices are on the outer face’s boundary. In such an embedding, denoted as H-embedding, the outer face forms a Hamiltonian cycle  $H_G=(x_1,x_2,\ldots,x_n)$, and all other faces are inner faces. Moreover, such an embedding is called an outerplane graph.
  A graph is outerplanar if and only if it contains neither $K_4$ nor $K_{2,3}$ minor.

  In 1958, Gr\"{o}tzsch et al. {\cite{Grotzsch1958}} established the renowned G\"{o}rtzsch's Theorem, which asserts that every triangle-free planar graph is 3-colorable. Subsequently, Gr\"{u}nbaum \cite{BrankoGrunbaum1963} and Aksionov \cite{Aksionov1974} extended this result by proving that planar graphs containing at most three triangles also admit 3-colorings. Their work further elucidated the extendability of colorings for 4- and 5-faces, marking a pivotal advancement in the field. This progress shifted research focus toward investigating 3-colorable planar graphs with triangular structures.

     La et al.\cite{LaLuSt} studied the 3-coloring extension under the constraint of planar graphs with at most one triangle and obtained the following results:  Theorem \ref{Hong1} and Theorem \ref{Hong2}.
\begin{theorem}[\cite{LaLuSt}]\label{Hong1}For any two independent vertices in a plane graph $G,$ any precoloring assigned to them  can be extended to a proper 3-coloring of the entire graph $G$.
\end{theorem}
\begin{theorem}[\cite{LaLuSt}]\label{Hong3}If $G$ is a plane graph with at most one triangle and $v$ is a vertex of degree no more than 3 in a graph $Q$ such that $G=Q-v$. When $G$ contains at most one triangle and $v$ is adjacent to at most two of its vertices. Then $Q$ is 3-colorable.
\end{theorem}

A graph $G$ is $K_4^{'}$-free if it contains no subdivision of $K_4$ where exactly three edges incident to one vertex are subdivided, and the central vertex is degree 3 in $G$. For a proper 3-coloring of $K_4^{'}$ (See Feagure \ref{K4^{'}}), $v’s$ three neighbors cannot share the same color, as the subdivided edges enforce non-monochromatic on the neighbors of $v$ constraints on the peripheral vertices.

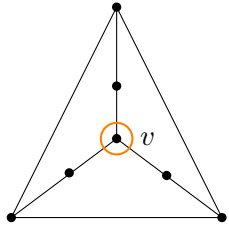
\begin{figure}
	\centering
	\begin{minipage}[htpb]{0.2\linewidth}
		\begin{tikzpicture}[scale=0.7]
			\fill (0,0) circle(2.5pt);
			\node[right=5pt] at (0,0) {$v$};
			\fill (0,1) circle (2.5pt);
			\node at (0,1) {};
			\fill (0,2.5) circle (2.5pt);
			\node at (0,2.5) {};
			\fill (-2,-1.5) circle (2.5pt);
			\node at (-2,-1.5) {};
			\fill (2,-1.5) circle (2.5pt);
			\node at (2,-1.5) {};
			\draw (0,2.5)--(0,0)--(-2,-1.5)--(2,-1.5)--(0,0) node (v1) {};
			\draw (-2,-1.5)--(0,2.5)--(2,-1.5);
			\fill (-0.9,-0.65) circle (2.5pt);
			\node at (-0.9,-0.65) {};
			\fill (0.95,-0.7) circle (2.5pt);
			\node at (0.95,-0.7) {};
			\draw[thick,orange]  (v1) ellipse (0.3 and 0.3);
		\end{tikzpicture}
	\end{minipage}
	\caption{The graph of $K_4^{'}$.}
	\label{K4^{'}}
\end{figure}
\begin{theorem}[\cite{LaLuSt}]\label{Hong2}
	If a $K_4^{'}$-free plane graph $G$ contains no more than one triangle, a precoloring with the same color of the neighbors of every vertex of degree at most 3  can be extended to a 3-coloring of $G$.
\end{theorem}

In any 3-coloring of a 5-face, the color distribution follows a specific pattern: two vertices are assigned to one color, two to another, and the remaining vertex, the special vertex, receives the third color. The edge directly opposite this special vertex within the 5-face is designated as the special edge.

\begin{theorem}[\cite{Aksionov1974,BrankoGrunbaum1963}]\label{Grum and Ak}
	Let $G$ be a plane graph. The following results hold.
	\item [\rm(i)] If $G$ has at most three triangles, then the chromatic number of
    $G$ is at most three;
	\item [\rm(ii)] If $G$ has at most one triangle, then
	\begin{itemize}
		\item \indent
		if $G$ has a 4-face $C$, every 3-coloring   of $C$ can be extended to a 3-coloring of $G$;
		\item \indent
		if $G$ has a 5-face $F$, every 3-coloring   of $F$ can be extended to a 3-coloring of $G$ if the special edge of $F$ is not contained in a triangle.
	\end{itemize}
\end{theorem}
La et al.\ raised the three  {\bf Problems} 4.2, 4.3 and 4.4 at the end of the article~\cite{LaLuSt}.
Problem 4.4~\cite{LaLuSt} has been addressed in \cite{Aksionov1974,BrankoGrunbaum1963} for 5-face.  In this paper, we attempt to address the first two problems in biconnected outerplanar graphs and then extend the results to outerplanar graphs.

In this study, we propose a structure composed of two triangles sharing a common edge, which we designate as the ``$diamond\ D$'' (Figure \ref{daimond}). Within this configuration, two independent vertices u and v are constrained to share the same color in any proper 3-coloring. These vertices are termed ``diamond vertices." Notably, in Problem 4.2 ~\cite{LaLuSt}, the presence of the diamond structure enforces that $u$ and $v$ must adopt identical colors; otherwise, the graph admits no proper 3-coloring.

\begin{figure}
	\begin{minipage}[h]{0.3\linewidth}
		\begin{tikzpicture}[scale=0.5]
			\draw  (0,0) node (v1) {} circle (100pt);
			\fill (0,2) circle(2.5pt) ;
			\node [above=3pt] at (0,2) {$u$};
			\fill (-2,0) circle(2.5pt) ;
			\node [above=3pt] at (-2,0) {};
			\fill (2,0) circle(2.5pt) ;
			\node [above=3pt] at (2,0) {};
			\fill (0,-2) circle(2.5pt) ;
			\node [below=3pt] at (0,-2) {$v$};
			\draw (0,2)--(-2,0)--(2,0)--(0,2);
			\draw (0,-2)--(-2,0)--(2,0)--(0,-2);
			\draw[thick,densely dashed](-2.4,2.55)--(0.5,3.5);
			\node at (-2.4,2.55) {};
			\node at (0.5,3.5) {};
			\draw[thick,densely dashed](-3.05,1.65) --(1.6,3.05);
			\node at (-3.05,1.65) {};
			\node at (1.6,3.05) {};
			\draw[thick,densely dashed](-3.45,0.65) --(-0.4,1.6);
			\node at (-3.45,0.65) {};
			\node at (-0.4,1.6) {};
			\draw[thick,densely dashed](0.25,1.7) --(2.45,2.4);
			\node at (0.25,1.7) {};
			\node at (2.45,2.4) {};
			\draw[thick,densely dashed](-3.5,-0.35) --(-1.9,0.15);
			\node at (-3.5,-0.35) {};
			\node at (-1.9,0.15) {};
			\draw[thick,densely dashed](1,0.95) --(3.05,1.7);
			\node at (1,0.95) {};
			\node at (3.05,1.7) {};
			\node (v2) at (-3.25,-1.35) {};
			\node (v3) at (-1.3,-0.7) {};
			\draw[thick,densely dashed]  (v2) edge (v3);
			\node (v6) at (-2.85,-2.15) {};
			\node (v7) at (-0.6,-1.4) {};
			\node (v4) at (1.75,0.25) {};
			\node (v5) at (3.45,0.9) {};
			\draw[thick,densely dashed]  (v4) edge (v5);
			\draw[thick,densely dashed]  (v6) edge (v7);
			\node (v8) at (1,-1) {};
			\node (v9) at (3.5,-0.05) {};
			\draw[thick,densely dashed]  (v8) edge (v9);
			\node (v10) at (-2.2,-2.85) {};
			\node (v11) at (0,-2.1) {};
			\draw  [thick,densely dashed](v10) edge (v11);
			\node (v12) at (3.45,-0.9) {};
			\draw[thick,densely dashed]  (v11) edge (v12);
			\node (v13) at (-1.05,-3.35) {};
			\node (v14) at (2.9,-2.05) {};
			\draw[thick,densely dashed] (v13) edge (v14);
		\end{tikzpicture}
	\end{minipage}
	\centering
	\caption{The structure of ``$diamond\ D$''.}
	\label{daimond}
\end{figure}
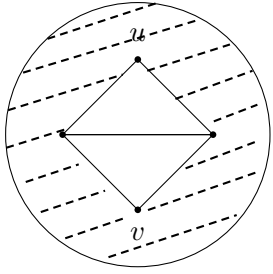
We establish new results on 3-precoloring extension properties for biconnected outerplanar graphs with at most two triangles and partially precolored vertices, formalized in Theorems \ref{prove1} and \ref{prove2}.

\begin{theorem}\label{prove1}
Given a biconnected outerplane graph $G$ containing a single triangle, any precoloring of three independent vertices can be extended to a 3-coloring of $G$.
\end{theorem}
\begin{theorem}\label{prove2}
A biconnected outerplane graph $G$ containing exactly two triangules satisfies the property that every 3-coloring assignment to any pair of independent vertices can be extended to a 3-coloring of $G$ under the following three conditions.
   \item [\rm(a)] $G$ forbids the diamond ``$diamond\ D$''
 as a subgraph.
   \item [\rm(b)]  $G$ admits a diamond $D$ as a subgraph, but any two precolored nonadjacent vertices share at most one vertex of $D$.
   \item [\rm(c)] $G$ admits ``$diamond\ D$'' and the precolored two nonadjacent vertices are ``$diamond\ vertices$'', then they must have the same precolored color.

\end{theorem}
Since 1-connected outerplane graphs are subgraphs of certain 2-connected outerplane graphs obtained by edge additions, they share the same number of triangles.

\begin{corollary}Consider a 1-connected outerplane graph $G$ containing precisely one triangle. For any subset of three independent vertices in $G$, every proper precoloring of these vertices admits a completion to a 3-coloring of $G$."
\end{corollary}

\begin{corollary}Consider a 1-connected outerplane graph $G$ containing precisely two triangless. For any subset of two independent vertices in $G$, every proper precoloring of these vertices admits a completion to a 3-coloring of $G$ in the following three cases.
   \item [\rm(a)] $G$ forbids the diamond ``$diamond\ D$''
 as a subgraph.
   \item [\rm(b)] $G$ admits a diamond $D$ as a subgraph, but any two precolored nonadjacent vertices share at most one vertex of $D$.
   \item [\rm(c)] $G$ admits  ``$diamond\ D$'' and the precolored two nonadjacent vertices are ``$diamond\ vertices$'', then they must have the same precolored color.
\end{corollary}
\section{Preliminaries}

In this contribution, only finite, undirected, and simple graphs are taken into account. Consider a graph $G=(V,E)$, where the number of vertices and edges are denoted by $n$ and $m$ respectively.The degree $d_G(x)$ (also simply written as $d(x)$) of a vertex $x$ is defined as the number of its neighbors in $G$. For any two vertices $x$ and $y$ in $G$, the distance between them, denoted as $d_G(x,y)$ (or $
d(x,y)$), is the length of the shortest path connecting $x$ and $y$.
The open - neighborhood of a vertex $x$ in $G$ is defined as $N(x)$=$\{y\ |\ xy\in E(G)\}$. Let $N_{R_i}(x)$ be the set of colored neighbors of the vertex $x$ within the face $R_i.$
 Let $F_G$
  denote the set of all faces of
 $G$, and $f$ represent the number of faces. A face with $k$ edges is called a $k$-face.




\begin{observation}
	If $F$ is an $n$-$face\ (n\ge6)$ of an outerplane graph $G$, then $F$ can be decomposed into 4-faces and 5-faces through adding chords.
\end{observation}

\begin{observation}\label{observation2}
	For any given outerplane graph $G$ with at most two triangles of order $n\ge6$, we add some chords and get a new graph $H$. Any face in the graph $H$ has length no more than 5. It is clear that the graph $G$ is a subgraph of $H$ and the method of adding chords is not unique.
\end{observation}


In the $H$ - embedding of graph $G$, a triangle that lies completely within the interior of the embedding is defined as an inner triangle. An internal triangle does not share any edge with the outer boundary of the $H$-embedding $H_G$. If an inner triangle shares one edge with $H_G$, it is termed a striped triangle. Finally, when an inner triangle shares two edges with $H_G$, it is called a marginal triangle.

A proper coloring of a graph $G$ ensures that adjacent vertices receive distinct colors. A graph is $k$-colorable if it can be colored with at most
$k$ colors under this constraint. The chromatic number $\chi(G)$ denotes the minimal $k$ for which such a coloring exists.



The precoloring extension problem concerns the determination of whether a partial coloring of a graph can be consistently expanded to a proper $k$-coloring of the entire graph. Given a graph $G=(V,E)$, a vertex subset
$W \subseteq V$, and a preassigned coloring function $c: W \rightarrow \{1, 2, ..., k\}$, does there exist a proper $k$-coloring $c'$ of $G$ such that such that $c'(v) = c(v)$ for all $v\in W$?~\cite{Precoloring}
\begin{figure}
	\begin{minipage}[H]{0.5\linewidth}
		\begin{tikzpicture}[scale=0.4]
			\fill (-3.86,4.12) circle(2.5pt) ;
			\node [above=4pt] at (-3.86,4.12) {$a$};
			\fill (-7.14,1.72) circle(2.5pt) ;
			\node [left=4pt] at (-7.14,1.72) {$e$};
			\fill (-0.56,1.74) circle(2.5pt) ;
			\node [right=4pt] at (-0.56,1.74) {$b$};
			\fill (-5.87,-2.14) circle(2.5pt) ;
			\node [below=4pt] at (-5.87,-2.14) {$d$};
			\fill (-1.81,-2.13) circle(2.5pt) ;
			\node [below=4pt] at (-1.81,-2.13) {$c$};
			\fill (4,2) circle(2.5pt) ;
			\node [above=4pt] at (4,2) {$x$};
			\fill (4,-2) circle(2.5pt) ;
			\node [below=4pt] at (4,-2) {$y$};
			\fill (11,-2) circle(2.5pt) ;
			\node [below=4pt] at (11,-2) {$z$};
			\draw (-3.86,4.12)--(-7.14,1.72) --(-5.87,-2.14)--(-1.81,-2.13)--(-0.56,1.74)--(-3.86,4.12);
			\draw (4,2)--(4,-2)--(11,-2)--(4,2);
			\node [below=2pt] at (-3.84,-3.13) {$Graph\ G$};
			\node [below=2pt] at (7.5,-3) {$Graph\ G'$};
		\end{tikzpicture}
	\end{minipage}
	\centering
	\caption{A mapping function \( f \) satisfies \( f(a) = x \), \( f(b) = y \), \( f(c) = z \), \( f(d) = y \), \( f(e) = z \), such that graphs \( G \) and \( G' \) are homomorphic.}
	\label{homomorphism}
\end{figure}
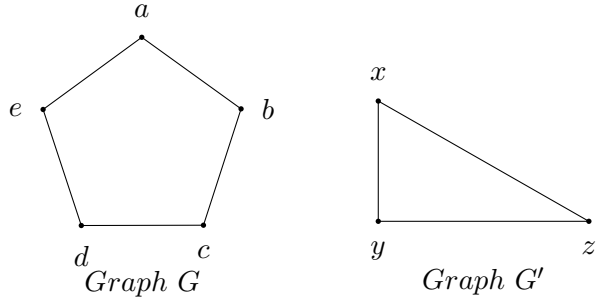

A graph homomorphism  $f: G \rightarrow H$(where $G=(V_G,E_G), H=(V_H,E_H)$) is a vertex mapping $f:V_G\rightarrow V_H.$ satisfying the edge preservation condition: for all $(u,v)\in E_G$, $(f(u),f(v))\in E_H.$  Notably, homomorphisms exist from both a 4-cycle and a 5-cycle to \( K_3 \) (Figure \ref{homomorphism}).

\section{ Verification of Theorems \ref{prove1} and \ref{prove2}}
The proof of Theorem \ref{prove1} is divided into three auxiliary theorems. We first tackle the case where the three precolored vertices are assigned distinct colors, which corresponds to Theorem \ref{3.1}.
\begin{theorem}\label{3.1}
	Given a biconnected outerplane graph $G$ with one triangle. If any three independent vertices are precolored with different colors, then the coloring can be extended to a proper 3-coloring of $G$.
\end{theorem}
\begin{proof}
	Assume that any precolored three independent vertices of $G$ are $y_1$, $y_2$ and $y_3$ in clockwise order and $d(y_1,y_2)\ge2$, $d(y_2,y_3)\ge2$. Let $F$ be the outer face of $G$. From the plane graph $G$, construct $G'$ by inserting vertices $u,v$ and edges $y_1u$, $uy_2$, $y_2v$, $vy_3$, $y_1y_3$ (Figure \ref{figure3.1}). Here $G'$ can not produce new triangles, because $d(y_1,y_2)\ge2$, $d(y_2,y_3)\ge2$.
	In $G'$ color the vertices of the 5-face $C=y_1uy_2vy_3$ by 1,3,2,1,3, respectively. Observe that the special edge $uy_2$ is not in the triangle. By Theorem \ref{Grum and Ak} (ii), the 3-coloring of $C$ can be extended to a 3-coloring of $G'$, which induces a 3-coloring of $G$ in which $y_1$, $y_2$ and $y_3$ are colored differently.
\end{proof}
\begin{figure}[h!]
	\centering
	\begin{minipage}[H]{0.4\linewidth}
		\begin{tikzpicture}[scale=0.5]
			\fill (-3,5) circle(4.5pt) ;
			\node [right=4pt] at (-3,5) {$y_1$};
			\fill (6,5) circle(4.5pt) ;
			\node [left=4pt] at (6,5) {1\  \ $y_1$};
			\fill (-3,3) circle(2.5pt) ;
			\node [left=4pt] at (-3,3) {};
			\fill (6,3) circle(2.5pt) ;
			\node [left=4pt] at (6,3) {};
			\fill (-3,1) circle(4.5pt) ;
			\node [right=4pt] at (-3,1) {$y_2$};
			\fill (6,1) circle(4.5pt) ;
			\node [left=4pt] at (6,1) {2\  \ $y_2$};
			\fill (-3,-1) circle(2.5pt) ;
			\node [left=4pt] at (-3,-1) {};
			\fill (6,-1) circle(2.5pt) ;
			\node [left=4pt] at (6,-1) {};
			\fill (-3,-3) circle(4.5pt) ;
			\node [right=4pt] (v1) at (-3,-3) {$y_3$};
			\fill (6,-3) circle(4.5pt) ;
			\node [left=4pt] (v1) at (6,-3) {3\  \ $y_3$};
			\node [right=4pt] (v3) at (-3,-4) {};
			\node [right=4pt] (v3) at (6,-4) {};
			\node [left=4pt] at (-3,6) {};
			\node [left=4pt] at (-3.9,5.68) {};
			\node [left=4pt] at (5.1,5.68) {};
			\node [left=4pt] (v2) at (-3.98,-3.5) {};
			\node [left=4pt] (v2) at (5.02,-3.5){};
			\node [left=4pt] at (-3.5,1) {};
			\node [left=4pt] at (5.5,1) {};
			\draw (-3,6)--(-3,5)--(-3,3)--(-3,1)--(-3,-1)--(-3,-3)--(-3,-4);
			\draw (6,6)--(6,5)--(6,3)--(6,1)--(6,-1)--(6,-3)--(6,-4);
			\draw  (-3,-3)--(-3.98,-3.5);
			\draw(6,-3) node (v5) {}--(5.02,-3.5);
			\draw  (-3,1)--(-3.5,1);
			\draw  (6,1) node (v7) {}--(5.5,1);
			\draw (-3,5)--(-3.9,5.68);
			\draw(6,5) node (v4) {}--(5.1,5.68);
			\node at (0,1) {\Large$F$};
			\node[right=4pt] (v6) at (6.5,3) {3};
			\node[right=4pt] (v6) at (6.5,-1) {1};
			\fill (8,3) circle(4.5pt) ;
			\node[right=4pt] (v6) at (8,3) {$u$};
			\fill (8,-1) circle(4.5pt) ;
			\node[right=4pt] (v8) at (8,-1) {$v$};
			\draw[thick,green] (6,5)  arc(90:-90:4);
			\draw[thick,green] (6,5)  arc(90:-90:2);
			\draw[thick,green] (6,1)  arc(90:-90:2);
			\node at (8.5,1) {\Large$C$};
		\end{tikzpicture}
	\end{minipage}
	\caption{}
	\label{figure3.1}
\end{figure}
In Theorem \ref{prove1}'s proof (Part II), it is demonstrated that assigning the same color is feasible for any three pairwise independent vertices within a biconnected outerplane graph possessing exactly one triangle.

\begin{theorem}\label{3.2}
Consider a biconnected outerplane graph $G$ with a single triangle. Given that any three independent vertices are precolored identically, the coloring admits an extension to a proper 3-coloring of $G$.	
\end{theorem}
\begin{proof}
A biconnected outerplane graph $G$ containing exactly one triangle does not admit any $K_4^{'}$ subgraphs as induced subgraphs.  $y_1, y_2, y_3$ are the independent three vertices precolored by the same color. Let $F$ be the outer face of $G$. From $G$, we obtain a planar graph $G^1$ by inserting in the interior of $F$ a vertex $v$ and edges $vy_1, vy_2, vy_3$. Here $G^1$ can not increase the number of triangles, because $y_1, y_2, y_3$ are the independent three vertices. In $G^1$, the degree of $v$ is 3 and its neighbors are independent and colored the same.
    By Theorem \ref{Hong2}, the coloring propagates to a proper 3-coloring of $G$, ensuring chromatic consistency across all precolored vertices.

\end{proof}
\begin{theorem}\label{3.3}
For any biconnected outerplane graph $G$ containing exactly one triangle, if a partial 3-coloring is given such that: 1.Three independent vertices are pre-colored,
2. Two vertices share one color while the third uses a distinct color,
then this partial coloring can be extended to a proper 3-coloring of the entire graph $G$.
\end{theorem}
\begin{proof}
	Assume that any precolored three independent vertices of $G$ are $y_1$, $y_2$ and $y_3$ in clockwise order and $d(y_1,y_2)\ge2$, $d(y_2,y_3)\ge2$. Since two independent vertices have the same color, we only need to discuss two cases (because $y_1$, $y_2$, and $y_3$ are equivalent, other similar cases are not discussed).	
	\setcounter{case}{0}
	\begin{case}
		When the vertices $y_1$ and $y_2$ are precolored by the same color.
	\end{case}
Assume $F$ is the unbounded face of $G$. By augmenting
$G$ with two new vertices $u$ and $v$ and edges $y_1u$, $uy_2$, $y_2v$, $vy_3$, $y_1y_3$ (Figure \ref{insert uv})
 , we derive a plane graph $G^2$. Here $G^2$ can not produce new triangles, because $d(y_1,y_2)\ge2$, $d(y_2,y_3)\ge2$.
	In $G^2$ color the vertices of the 5-face $C=y_1uy_2vy_3$ by 1,2,1,2,3, respectively. Observe that the special edge $uy_2$ is not contained in any triangle. By Theorem \ref{Grum and Ak} (ii), the 3-coloring of $C$ can be extended to a 3-coloring of $G^2$, which induces a 3-coloring of $G$ in which $y_1$, $y_2$ are colored by the same color, while the vertex $y_3$ is colored differently from them.
	\begin{figure}[h!]
		\centering
		\begin{minipage}[H]{0.4\linewidth}
			\begin{tikzpicture}[scale=0.5]
				\fill (-3,5) circle(4.5pt) ;
				\node [right=4pt] at (-3,5) {$y_1$};
				\fill (6,5) circle(4.5pt) ;
				\node [left=4pt] at (6,5) {1\ \ $y_1$};
				\fill (-3,3) circle(2.5pt) ;
				\node [left=4pt] at (-3,3) {};
				\fill (6,3) circle(2.5pt) ;
				\node [left=4pt] at (6,3) {};
				\fill (-3,1) circle(4.5pt) ;
				\node [right=4pt] at (-3,1) {$y_2$};
				\fill (6,1) circle(4.5pt) ;
				\node [left=4pt] at (6,1) {1\ \ $y_2$};
				\fill (-3,-1) circle(2.5pt) ;
				\node [left=4pt] at (-3,-1) {};
				\fill (6,-1) circle(2.5pt) ;
				\node [left=4pt] at (6,-1) {};
				\fill (-3,-3) circle(4.5pt) ;
				\node [right=4pt] (v1) at (-3,-3) {$y_3$};
				\fill (6,-3) circle(4.5pt) ;
				\node [left=4pt] (v1) at (6,-3) {3\ \ $y_3$};
				\node [right=4pt] (v3) at (-3,-4) {};
				\node [right=4pt] (v3) at (6,-4) {};
				\node [left=4pt] at (-3,6) {};
				\node [left=4pt] at (-3.9,5.68) {};
				\node [left=4pt] at (5.1,5.68) {};
				\node [left=4pt] (v2) at (-3.98,-3.5) {};
				\node [left=4pt] (v2) at (5.02,-3.5){};
				\node [left=4pt] at (-3.5,1) {};
				\node [left=4pt] at (5.5,1) {};
				\draw (-3,6)--(-3,5)--(-3,3)--(-3,1)--(-3,-1)--(-3,-3)--(-3,-4);
				\draw (6,6)--(6,5)--(6,3)--(6,1)--(6,-1)--(6,-3)--(6,-4);
				\draw  (-3,-3)--(-3.98,-3.5);
				\draw(6,-3) node (v5) {}--(5.02,-3.5);
				\draw  (-3,1)--(-3.5,1);
				\draw  (6,1) node (v7) {}--(5.5,1);
				\draw (-3,5)--(-3.9,5.68);
				\draw(6,5) node (v4) {}--(5.1,5.68);
				\node at (0,1) {\Large$F$};
				\fill (8,3) circle(4.5pt) ;
				\node[right=4pt] (v6) at (8,3) {$u$};
				\fill (8,-1) circle(4.5pt) ;
				\node[right=4pt] (v8) at (8,-1) {$v$};
				\draw[thick,red] (6,5)  arc(90:-90:4);
				\draw[thick,red] (6,5)  arc(90:-90:2);
				\draw[thick,red] (6,1)  arc(90:-90:2);
				\node at (8.5,1) {\Large$C$};
				\node[right=4pt] (v6) at (6.5,3) {2};
				\node[right=4pt] (v6) at (6.5,-1) {2};
			\end{tikzpicture}
		\end{minipage}
		\caption{}
		\label{insert uv}
	\end{figure}
	\begin{case}
		When the vertices $y_1$ and $y_3$ are precolored by the same color.
	\end{case}
	If $d(y_1,y_2)=2,\ d(y_2,y_3)=2$, let $p$ be the neighbor of $y_1$ and $y_2$, and $q$ be the neighbor of $y_2$ and $y_3$ and $F$ be the outer face of $G$.
 From the plane graph $G$, construct $G^3$ by inserting vertices $u,v$ and edges $y_1y_2$, $y_2y_3$, $y_1u$, $uv$, $vy_3$ (Figure \ref{insert uv1}). If $p$ and $q$ are adjacent vertices, then the triangle in $G$ is $T=py_2q$. In this structure, $y_1,\ y_3$ cannot have the same color, so $p$ and $q$ are nonadjacent vertices. From $G^3$, we obtain a plane graph $G^{3'}$ by deleting $p$ and $q$, color the 5-cycle $C'=y_1y_2y_3uv$ with 3 colors such that $y_1$, $y_3$ have the same color. By Theorem \ref{Grum and Ak} (ii), the 3-coloring of $C'$ can be propagated to a 3-oloring of $G^{3'}$, thereby enforcing that $y_1$ and $y_3$ share a common color while $y_2$ assumes a distinct one in the induced 3-coloring of $G$.

		\begin{figure}[htbp]
		\centering
		\begin{minipage}[H]{0.7\linewidth}
			\begin{tikzpicture}[scale=0.5]
				\fill (-3,5) circle(4.5pt) ;
				\node [right=4pt] at (-3,5) {$y_1$};
				\fill (6,5) circle(4.5pt) ;
				\node [left=4pt] at (6,5) {1\ \ $y_1$};
				\fill (15,5) circle(4.5pt) ;
				\node [left=4pt] at (15,5) {1\ \ $y_1$};
				\fill (-3,3) circle(2.5pt) ;
				\node [left=4pt] at (-3,3) {             $p$};
				\fill (6,3) circle(2.5pt) ;
				\node [left=4pt] at (6,3) {             $p$};
				\fill (-3,1) circle(4.5pt) ;
				\node [right=4pt] at (-3,1) {$y_2$};
				\fill (6,1) circle(4.5pt) ;
				\node [left=4pt] at (6,1) {2\ \ $y_2$};
				\fill (15,1) circle(4.5pt) ;
				\node [left=4pt] at (15,1) {2\ \ $y_2$};
				
				\fill (-3,-1) circle(2.5pt) ;
				\node [left=4pt] at (-3,-1) {             $q$};
				\fill (6,-1) circle(2.5pt) ;
				\node [left=4pt] at (6,-1) {             $q$};
				
				\fill (-3,-3) circle(4.5pt) ;
				\node [right=4pt] (v1) at (-3,-3) {$y_3$};
				\fill (6,-3) circle(4.5pt) ;
				\node [left=4pt] (v1) at (6,-3) {1\ \ $y_3$};
				\fill (15,-3) circle(4.5pt) ;
				\node [left=4pt] (v1) at (15,-3) {1\ \ $y_3$};
				\node [right=4pt] (v3) at (-3,-4) {};
				\node [right=4pt] (v3) at (6,-4) {};
				\node [left=4pt] at (-3,6) {};
				\node [left=4pt] at (-3.9,5.68) {};
				\node [left=4pt] at (5.1,5.68) {};
				\node [left=4pt] (v2) at (-3.98,-3.5) {};
				\node [left=4pt] (v2) at (5.02,-3.5){};
				\node [left=4pt] at (-3.5,1) {};
				\node [left=4pt] at (5.5,1) {};
				\draw (-3,6)--(-3,5)--(-3,3)--(-3,1)--(-3,-1)--(-3,-3)--(-3,-4);
				\draw (6,6)--(6,5)--(6,3)--(6,1)--(6,-1)--(6,-3)--(6,-4);
				\draw (15,6)--(15,5); \draw(15,-3)--(15,-4);
				\draw  (-3,-3)--(-3.98,-3.5);
				\draw(6,-3) node (v5) {}--(5.02,-3.5);
				\draw(15,-3) node (v5) {}--(14.02,-3.5);
				\draw  (-3,1)--(-3.5,1);
				\draw  (6,1) node (v7) {}--(5.5,1);
				\draw  (15,1) node (v7) {}--(14.5,1);
				\draw (-3,5)--(-3.9,5.68);
				\draw(6,5) node (v4) {}--(5.1,5.68);
				\draw(15,5) node (v4) {}--(14.1,5.68);
				\node at (0,1) {\Large$F$};
				\fill(9.1,3.5) circle(4.5pt);
				\fill(9.05,-1.6) circle(4.5pt);
				\node[right=4pt] at (9.1,3.5) {$u$};
				\node[right=4pt] at (9.05,-1.6) {$v$};
				\fill(18.1,3.5) circle(4.5pt);
				\fill(18.05,-1.6) circle(4.5pt);
				\node[right=4pt] at (18.1,3.5) {$u$};
				\node[right=4pt] at (18.05,-1.6) {$v$};
				\draw[thick,purple] (6,5)  arc(90:-90:4);
				\draw[thick,purple] (6,5)  arc(90:-90:2);
				\draw[thick,purple] (6,1)  arc(90:-90:2);
				\node at (17.5,1) {\Large$C'$};
				\draw[thick,purple] (15,5)  arc(90:-90:4);
				\draw[thick,purple] (15,5)  arc(90:-90:2);
				\draw[thick,purple] (15,1)  arc(90:-90:2);
				\node[right=4pt] (v6) at (7.6,3.5) {2};
				\node[right=4pt] (v6) at (7.55,-1.6) {3};
				\node[right=4pt] (v6) at (16.6,3.5) {2};
				\node[right=4pt] (v6) at (16.55,-1.6) {3};
			\end{tikzpicture}
		\end{minipage}
		\caption{}
		\label{insert uv1}
	\end{figure}
	
	If $d(y_1,y_2)=2,\ d(y_2,y_3)>2$, from $G$, we obtain a plane graph $G^4$ by inserting two vertices $u$ and $v$ and edges $y_1u$, $uy_3$, $y_2y_3$, $y_2v$, $y_1v$ (Figure \ref{insert uv2}) in $F$. Due to $d(y_1,y_2)=2,\ d(y_2,y_3)>2$, no new triangles can be formed. Color the 5-cycle $C''=y_1vy_2y_3u$ in the graph $G^4$ using 3 colors. Since the added edges are outside the triangle and no special edge is created. By Theorem \ref{Grum and Ak} (ii), the 3-coloring of $C''$ can be extended to a 3-coloring of $G^4$, which induces a 3-coloring of $G$ in which $y_1$, $y_3$ are colored by the same color, while the vertex $y_2$ is colored by different color from them. For $d(y_1,y_2)>2,\ d(y_2,y_3)=2$, prove it using the same approach as above.
	\begin{figure}[h!]
		\centering
		\begin{minipage}[H]{0.4\linewidth}
			\begin{tikzpicture}[scale=0.5]
				\fill (-3,5) circle(4.5pt) ;
				\node [right=4pt] at (-3,5) {$y_1$};
				\fill (6,5) circle(4.5pt) ;
				\node [left=4pt] at (6,5) {1\ \ $y_1$};
				\fill (-3,1) circle(4.5pt) ;
				\node [right=4pt] at (-3,1) {$y_2$};
				\fill (6,1) circle(4.5pt) ;
				\node [left=15pt] at (6.9,1) {2\ \ $y_2$};
				\fill (-3,-0.5) circle(2.5pt) ;
				\node [left=4pt] at (-3,-0.5) {};
				\fill (6,-0.5) circle(2.5pt) ;
				\node [left=4pt] at (6,-0.5) {};
				\fill (-3,-3) circle(4.5pt) ;
				\node [right=4pt] (v1) at (-3,-3) {$y_3$};
				\fill (6,-3) circle(4.5pt) ;
				\node [left=4pt] (v1) at (6,-3) {1\ \ $y_3$};
				\node [right=4pt] (v3) at (-3,-4) {};
				\node [right=4pt] (v3) at (6,-4) {};
				\node [left=4pt] at (-3,6) {};
				\node [left=4pt] at (-3.9,5.68) {};
				\node [left=4pt] at (5.1,5.68) {};
				\node [left=4pt] (v2) at (-3.98,-3.5) {};
				\node [left=4pt] (v2) at (5.02,-3.5){};
				\node [left=4pt] at (-3.5,1) {};
				\node [left=4pt] at (5.5,1) {};
				\fill (-3,3) circle(2.5pt) ;
				\node [left=4pt]at (-3,3) {};
				\fill (6,3) circle(2.5pt) ;
				\node [left=4pt]at (6,3) {};
				\draw (-3,6)--(-3,5)--(-3,3)--(-3,1)--(-3,-0.5)--(-3,-3)--(-3,-4);
				\draw (6,6)--(6,5)--(6,3)--(6,1)--(6,-0.5)--(6,-3)--(6,-4);
				\draw  (-3,-3)--(-3.98,-3.5);
				\draw(6,-3) node (v5) {}--(5.02,-3.5);
				\draw  (-3,1)--(-3.5,1);
				\draw  (6,1) node (v7) {}--(5.5,1);
				\draw (-3,5)--(-3.9,5.68);
				\draw(6,5) node (v4) {}--(5.1,5.68);
				\node at (0,1) {\Large$F$};
				\draw[thick,blue] (6,5)  arc(90:-90:4);
				\draw[thick,blue] (6,5)  arc(90:-90:2);
				\draw[thick,blue] (6,1)  arc(90:-90:2);
				\node at (8.5,1) {\Large$C''$};
				\fill(-3,-1.5) circle(2.5pt);
				\node at (-3,-1.5) {};
				\fill(10,1) circle(4.5pt);
				\fill(6,-1.5) circle(2.5pt);
				\node at (6,-1.5) {};
				\node[right=4pt] at (10,1) {$u$};
				\fill(8,3) circle(4.5pt);
				\node[right=4pt] at (6.5,3) {3};
				\node[right=4pt] at (8.5,1) {3};
				\node[right=4pt] at (8,3) {$v$};
			\end{tikzpicture}
		\end{minipage}
		\caption{}
		\label{insert uv2}
	\end{figure}
	
	If $d(y_1,y_2)>2,\ d(y_2,y_3)>2$, from $G$, we obtain a plane graph $G^5$ by inserting two vertices $u$ and $v$ and edges $y_1u$, $y_1y_2$, $y_2y_3$, $y_3v$, $uv$ (Figure \ref{insert uv3}) in $F$. Here $G^5$ still contain only one triangle. In $G^5$ color the vertices of the 5-cycle $C'''=y_1uvy_3y_2$ with three colors. Observe that the special edge will never contain in the triangle. By Theorem \ref{Grum and Ak} (ii), the 3-coloring of $C'''$ can be extended to a 3-coloring of $G^5$, which induces a 3-coloring of $G$ in which $y_1$, $y_3$ are colored by the same color, while the vertex $y_2$ is colored differently from them.
	\begin{figure}[h!]
		\centering
		\begin{minipage}[H]{0.4\linewidth}
			\begin{tikzpicture}[scale=0.5]
				\fill (-3,5) circle(4.5pt) ;
				\node [right=4pt] at (-3,5) {$y_1$};
				\fill (6,5) circle(4.5pt) ;
				\node [left=4pt] at (6,5) {1\ \ $y_1$};
				\fill (-3,2.5) circle(2.5pt) ;
				\node [left=4pt] at (-3,2.5) {};
				\fill (6,2.5) circle(2.5pt) ;
				\node [left=4pt] at (6,2.5) {};
				\fill (-3,1) circle(4.5pt) ;
				\node [right=4pt] at (-3,1) {$y_2$};
				\fill (6.0,1) circle(4.5pt) ;
				\node [left=15pt] at (6.5,1) {2\ \ $y_2$};
				\fill (-3,-0.5) circle(2.5pt) ;
				\node [left=4pt] at (-3,-0.5) {};
				\fill (6,-0.5) circle(2.5pt) ;
				\node [left=4pt] at (6,-0.5) {};
				\fill (-3,-3) circle(4.5pt) ;
				\node [right=4pt] (v1) at (-3,-3) {$y_3$};
				\fill (6,-3) circle(4.5pt) ;
				\node [left=4pt] (v1) at (6,-3) {1\ \ $y_3$};
				\node [right=4pt] (v3) at (-3,-4) {};
				\node [right=4pt] (v3) at (6,-4) {};
				\node [left=4pt] at (-3,6) {};
				\node [left=4pt] at (-3.9,5.68) {};
				\node [left=4pt] at (5.1,5.68) {};
				\node [left=4pt] (v2) at (-3.98,-3.5) {};
				\node [left=4pt] (v2) at (5.02,-3.5){};
				\node [left=4pt] at (-3.5,1) {};
				\node [left=4pt] at (5.5,1) {};
				\draw (-3,6)--(-3,5)--(-3,3.5)--(-3,1)--(-3,-0.5)--(-3,-3)--(-3,-4);
				\draw (6,6)--(6,5)--(6,3.5)--(6,1)--(6,-0.5)--(6,-3)--(6,-4);
				\draw  (-3,-3)--(-3.98,-3.5);
				\draw(6,-3) node (v5) {}--(5.02,-3.5);
				\draw  (-3,1)--(-3.5,1);
				\draw  (6,1) node (v7) {}--(5.5,1);
				\draw (-3,5)--(-3.9,5.68);
				\draw(6,5) node (v4) {}--(5.1,5.68);
				\node at (0,1) {\Large$F$};
				\draw[thick,orange] (6,5)  arc(90:-90:4);
				\draw[thick,orange] (6,5)  arc(90:-90:2);
				\draw[thick,orange] (6,1)  arc(90:-90:2);
				\node at (8.5,1) {\Large$C'''$};
				\fill(-3,3.5) circle(2.5pt);
				\node at (-3,3.5) {};
				\fill(6,3.5) circle(2.5pt);
				\node at (6,3.5) {};
				\fill(-3,-1.5) circle(2.5pt);
				\node at (-3,-1.5) {};
				\fill(6,-1.5) circle(2.5pt);
				\node at (6,-1.5) {};
				\fill(9.1,3.5) circle(4.5pt);
				\fill(9.05,-1.6) circle(4.5pt);
				\node[right=4pt] at (9.1,3.5) {$u$};
				\node[right=4pt] at (9.05,-1.6) {$v$};
				\node[right=4pt] (v6) at (7.6,3.5) {2};
				\node[right=4pt] (v6) at (7.55,-1.6) {3};
			\end{tikzpicture}
		\end{minipage}
		\caption{}
		\label{insert uv3}
	\end{figure}
\end{proof}

The combined implications of Theorems \ref{3.1}, \ref{3.2}, and \ref{3.3} establish Theorem \ref{prove1}. The demonstration of Theorem \ref{prove2} involves two steps: its initial step addresses the case of distinctly colored precolored vertices, which aligns with Theorem 3.4.

\begin{theorem}\label{3.4}
If $G$ is a 2-connected outerplane graph with two triangles and  $Q$ is a graph containing two triangles such that $G=Q-e$ for some edge $e\in Q$. Then
$Q$ is 3-colorable.
\end{theorem}
\begin{proof}
	Let $G$ be a biconnected outerplane graph with two triangles, and let $u$ and $v$ be any two vertices in $G$ that are colored differently and nonadjacent. Let $e=uv$, and $G+e=H$. If $d_{G}(u,v)=2$, then $H$ is a plane graph whose facial structure includes precisely three 3-cycles. If $d_{G}(u,v)>2$, then $H$ is a plane graph with two triangles. Otherwise,   $H$ is non-planar,  since it will contain subgraphs isomorphic to $K_5$ or $K_{3,3}$ minor. Removing edge $e$ would lead to the existence of subgraphs isomorphic to $K_4$ or $K_{2,3}$ minor, which contradicts the fact that $G$ is a biconnected outerplane graph. Therefore, by Theorem \ref{Grum and Ak}, the graph $H$ is 3-colorable.
\end{proof}
In proving Theorem \ref{prove2} (Step 2), given a biconnected outerplane graph containing exactly two triangles, it is established that two arbitrary independent vertices admit identical coloring.


Before we begin the proof, we first define some structures and homomorphism mappings and propose an algorithm to prove Theorem \ref{3.5} in the following.
\begin{definition}
	In a biconnected outerplane graph containing two triangles, if one of the triangles is a marginal triangle, we define the following two structures (a) and (b):
	\item[\rm(i)] If another triangle is also a marginal triangle and both triangles share a common vertex, and they are adjacent to the same 4-face or 5-face, we refer to such a structure as a ``$cake$'' structure (shown as Figure 9 (a)).
	\item[\rm(ii)] If another triangle is not a marginal triangle  and they are adjacent to the same 4-face or 5-face, we refer to such a structure as a ``$hamburger$'' structure (shown as Figure 9 (b)).
	\begin{figure}[htbp]
		\begin{minipage}[H]{0.8\linewidth}
			\begin{tikzpicture}[scale=1.0]
	\fill (-1.85,4) circle(2.5pt) ;
	\node at (-1.85,4) {};
	\fill (-2.5,3.25) circle(2.5pt) ;
	\node at (-2.5,3.25) {};
	\fill (-1,3.5) circle(2.5pt) ;
	\node at (-1,3.5) {};
	\fill (0,4) circle(2.5pt) ;
	\node at (0,4) {};
	\fill (0.65,3.3) circle(2.5pt) ;
	\node at (0.65,3.3) {};
	\draw (-2.5,3.25)--(-1.85,4)--(-1,3.5)--(-2.5,3.25)--(-1,2.2)--(0.65,3.3);
	\draw (-1,3.5)--(0,4)--(0.65,3.3)--(-1,3.5);
	\fill  (-1,2.2) circle(2.5pt) ;
	\node at (-1,2.2) {};
	\fill (4.15,4) circle(2.5pt) ;
	\node at (4.15,4) {};
	\fill (3.5,3.25) circle(2.5pt) ;
	\node at (3.5,3.25) {};
	\fill (5,3.5) circle(2.5pt) ;
	\node at (5,3.5) {};
	\fill (6,4) circle(2.5pt) ;
	\node at (6,4) {};
	\fill (6.65,3.3) circle(2.5pt) ;
	\node at (6.65,3.3) {};
	\draw (3.5,3.25)--(4.15,4)--(5,3.5)--(3.5,3.25)--(3.5,2.2)--(6.65,2.2)--(6.65,3.3);
	\draw (5,3.5)--(6,4)--(6.65,3.3)--(5,3.5);
	\fill  (3.5,2.2) circle(2.5pt) ;
	\node at (3.5,2.2) {};
	\fill  (6.65,2.2) circle(2.5pt) ;
	\node at (6.65,2.2) {};
	\fill (-6,0.5) circle(2.5pt) ;
	\node at (-6,0.5) {};
	\fill (-7,-0.25) circle(2.5pt) ;
	\node at (-7,-0.25) {};
	\fill (-5,-0.25) circle(2.5pt) ;
	\node at (-5,-0.25) {};
	\fill (-7,-1.25) circle(2.5pt) ;
	\node at (-7,-1.25) {};
	\fill (-5,-1.25) circle(2.5pt) ;
	\node at (-5,-1.25) {};
	\fill (-6,-2) circle(2.5pt) ;
	\node at (-6,-2) {};
	\draw (-6,0.5)--(-7,-0.25)--(-5,-0.25)--(-6,0.5);
	\draw (-7,-0.25)--(-7,-1.25)--(-5,-1.25)--(-5,-0.25);
	\draw (-7,-1.25)--(-6,-2)--(-5,-1.25);
	\fill (-1,0.5) circle(2.5pt) ;
	\node at (-1,0.5) {};
	\fill (-2,-0.25) circle(2.5pt) ;
	\node at (-2,-0.25) {};
	\fill (0,-0.25) circle(2.5pt) ;
	\node at (0,-0.25) {};
	\fill (-2,-1.25) circle(2.5pt) ;
	\node at (-2,-1.25) {};
	\fill (0,-1.25) circle(2.5pt) ;
	\node at (0,-1.25) {};
	\fill (-1,-2) circle(2.5pt) ;
	\node at (-1,-2) {};
	\fill (-2,-0.75) circle (2.5pt);
	\node at (-2,-0.75) {};
	\draw (-1,0.5)--(-2,-0.25)-- (-2,-0.75)--(-2,-1.25)--(-1,-2)--(0,-1.25)--(0,-0.25)--(-1,0.5) ;
	\draw (-2,-0.25)--(0,-0.25);
	\draw (-2,-1.25)--(0,-1.25);
	\fill (4,0.5) circle(2.5pt) ;
	\node at (4,0.5) {};
	\fill (3,-0.25) circle(2.5pt) ;
	\node at (3,-0.25) {};
	\fill (5,-0.25) circle(2.5pt) ;
	\node at (5,-0.25) {};
	\fill (5,-2.05) circle(2.5pt);
	\node at (5,-2.05) {};
	\fill (5,-1.15) circle(2.5pt);
	\node at (5,-1.15) {};
	\fill (3.3,-1.5) circle(2.5pt);
	\node at (3.3,-1.5) {};
	\draw (4,0.5)--(3,-0.25)--(3.3,-1.5)--(5,-2.05)--(5,-1.15)--(5,-0.25)-- (4,0.5);
	\draw (5,-2.05)--(3,-0.25)--(5,-0.25);
	\fill (9,0.5) circle(2.5pt) ;
	\node at (9,0.5) {};
	\fill (8,-0.25) circle(2.5pt) ;
	\node at (8,-0.25) {};
	\fill (10,-0.25) circle(2.5pt) ;
	\node at (10,-0.25) {};
	\fill (10,-2.05) circle(2.5pt);
	\node at (10,-2.05) {};
	\fill (8.3,-1.5) circle(2.5pt);
	\node at (8.3,-1.5) {};
	\fill (10,-0.85) circle(2.5pt);
	\node at (10,-0.85) {};
	\fill (10,-1.45) circle(2.5pt);
	\node at (10,-1.45) {};
	\draw (9,0.5)--(8,-0.25)--(8.3,-1.5)--(10,-2.05)--(10,-1.45)--(10,-0.85)--(10,-0.25)-- (9,0.5);
	\draw (10,-2.05)--(8,-0.25)--(10,-0.25);
	\node at (1.5,1.5) {$(a)$};
	\node at (1.5,-3) {$(b)$};
\end{tikzpicture}
		\end{minipage}
		\caption{(a) represents a ``cake'' structure, while (b) represents  ``hamburger'' structures (it should be noted that the hamburger structure is not limited to what is shown in (b).}
	\end{figure}
\end{definition}
Since the 4-face (5-face) is homomorphic to $K_3$, we can define the following types of homomorphism mappings.
\begin{definition}\label{hom mapping}
	The three vertices of $K_3$ are denoted by $x$, $y$, $z$ and we color them by $X$, $Y$, $Z$, respectively.
	\item[\rm(i)]  Let the vertices of a 4-face be $a$, $b$, $c$, $d$ in the clockwise direction, and we give six homomorphism mappings which needed in our following proof:
	\begin{itemize}
		\item $f_1$: \( f_1(a) = x \), \( f_1(b) = y \), \( f_1(c) = z \), \( f_1(d) = y \);
		\item $f_2$: \( f_2(a) = x \), \( f_2(b) = z \), \( f_2(c) = y \), \( f_2(d) = z \);
		\item $f_3$: \( f_3(a) = x \), \( f_3(b) = y \), \( f_3(c) = x \), \( f_3(d) = z \);
		\item $f_4$: \( f_4(a) = x \), \( f_4(b) = y \), \( f_4(c) = x \), \( f_4(d) = y \);
		\item $f_5$: \( f_5(a) = x \), \( f_5(b) = z \), \( f_5(c) = x \), \( f_1(d) = z \);
		\item $f_6$: \( f_1(a) = x \), \( f_1(b) = z \), \( f_1(c) = x \), \( f_1(d) = y \).
	\end{itemize}
	\item[\rm(ii)]Let the vertices of a 5-face be $a$, $b$, $c$, $d$, $e$ in the clockwise direction, and we give three homomorphism mappings:
	\begin{itemize}
		\item $F_1$: \( F_1(a) = x \), \( F_1(b) = z \), \( F_1(c) = y \), \( F_1(d) = z \), \( F_1(e) = y \);
		\item $F_2$: \( F_2(a) = x \), \( F_2(b) = z \), \( F_2(c) = y \), \( F_2(d) = x \), \( F_2(e) = y \);
		\item $F_3$: \( F_3(a) = y \), \( F_3(b) = z \), \( F_3(c) = x \), \( F_3(d) = z \), \( F_3(e) = x \).
	\end{itemize}
\end{definition}
\begin{remark}\label{remark 3.1}
	When using these homomorphism mappings, according to the following orders:
	\item[\rm(i)] $f_1>f_2>f_3>f_4>f_5>f_6$.
	\item[\rm(ii)] $F_1>F_2>F_3$.
\end{remark}
\begin{operation}\label{hom operation}
A biconnected outerplane graph $G$ of order  $n\ge6$ with only two triangles and a definite H-embedding on the plane. The outer face boundary vertices of the H-embedding exhibit a clockwise cyclic order of $x_1\ldots x_n.$  Due to observation \ref{observation2}, $H$ is obtained by iteratively adding chords and any face in the graph $H$ has length no more than 5. Let $u=x_i$ and $v=x_j$ (where $1<i<j<n$) be two independent vertices in the graph $H$ which have a same color. Let the face containing the vertex $u$ be denoted as $R_1$, and perform a homomorphic mapping $f_i\ (1<i<6)$ or $F_j\ (1<j<3)$ on it. Denote the next face, where two vertices of an edge have been processed, as $R_2$. Continue this mapping process until reach the face containing the vertex $v$, denoting this face as $R_k$.
\end{operation}
\begin{breakablealgorithm}
	\caption{Color adjusting algorithm based on homomorphic mapping}
	\label{algorithm1}
	\renewcommand{\algorithmicrequire}{\textbf{Input:}}
	\renewcommand{\algorithmicensure}{\textbf{Output:}}
	\begin{algorithmic}[1]	
		\REQUIRE A biconnected outerplane graph $H$ of order $n\ge6$ which has two triangles with a definite H-embedding on the plane. The outer face boundary vertices of the H-embedding exhibit a clockwise cyclic order of $x_1\ldots x_n.$ The two independent vertices $x_i$ and $x_j$ (where $1\leq i<j\leq n$) have the same color and all faces from $R_1$ to $R_k$ of $H$  after Operation \ref{hom operation}.
		\ENSURE $\psi(u)=\psi(v)\ (\psi:V(H)\rightarrow\{X,Y,Z\})$
		\STATE Perform a homomorphic mapping $f_i$ ($i=1,2$) or $F_1$ on the face $R_1$ containing vertex $u$, coloring it with the corresponding color. Then, continue to perform the homomorphic mapping on the next face that already has two vertices of one edge been colored, ensuring that it receives a proper coloring. The sequence of homomorphic mappings follows Remark \ref{3.1}, until reaching face $R_k$.
		\IF{$\psi(u)\notin\psi(N_{R_k}(v))$}
		\STATE $\psi(u)=\psi(v)$
		\ELSE
		\STATE Adjust the homomorphic mapping received by face $R_{k-1}$ based on its type, ensuring that $\psi(u)\notin\psi(N_{R_k}(v))$.
		\ENDIF
		\RETURN $\psi(u)=\psi(v)$
	\end{algorithmic}
\end{breakablealgorithm}
\begin{remark}
	\item[\rm(i)] The graph $H$ in the Algorithm \ref{algorithm1} satisfies the conditions of $H$ in Observation \ref{observation2}.
	\item[\rm(ii)] Because the operations to change the homomorphic mappings are finite and can be completed in linear time, the algorithm can process in linear time.
\end{remark}
\textbf{Explanation on Algorithm 1}

Due to the order of homomorphisms followed in Remark \ref{remark 3.1}, if the face $R_{k-1}$ is a 4-face, then there exists a homomorphism $f_1:R_{k-1}\rightarrow K_3$ or a homomorphism $f_2:R_{k-1}\rightarrow K_3$. The coloring scheme under the mapping $f_1$ (or $f_2$) is recorded as $\psi_1$ (or $\psi_2$). When the coloring of $R_{k-1}$ is obtained under the homomorphism $f_1$, if $\psi_1(u)=X$, then $v$ can also be colored by $X$, the algorithm ends. If $\psi_1(u)=Y$, then by changing the homomorphism on face $R_{k-1}$ from $f_1$ to $f_3$, $v$ can be colored by $Y$, the algorithm ends. If $\psi_1(u)=Z$, then by changing the homomorphism on face $R_{k-1}$ from $f_1$ to $f_4$, $v$ can be colored by $Z$, and the conclusion holds; the algorithm ends. When the coloring of $R_{k-1}$ is obtained under the homomorphism $f_2$, if $\psi_2(u)=X$, then $v$ can also be colored by $X$, the algorithm ends. If $\psi_2(u)=Y$, then by changing the homomorphism on face $R_{k-1}$ from $f_2$ to $f_5$, $v$ can be colored by $Y$, the algorithm ends. If $\psi_2(u)=Z$, then by changing the homomorphism on face $R_{k-1}$ from $f_2$ to $f_6$, $v$ can be colored by $Z$, the algorithm ends; if the face $R_{k-1}$ is a 5-face, then there exists a homomorphism $F_1:R_{k-1}\rightarrow K_3$. The coloring scheme under the mapping $F_1$ is recorded as $\psi_3$. When the face $R_{k-1}$ is colored under the homomorphism $F_1$, if $\psi_3(u)=X$, then $v$ can also be colored by $X$, the algorithm ends. If $\psi_3(u)=Y$, then by changing the homomorphism on face $R_{k-1}$ from $F_1$ to $F_3$, $v$ can be colored by $Y$, the algorithm ends. If $\psi_3(u)=Z$, then by changing the homomorphism on face $R_{k-1}$ from $F_1$ to $F_2$, $v$ can be colored by $Z$, the algorithm ends. Since the number of times the homomorphism can be changed is limited, the algorithm can run in linear time.
\begin{theorem}\label{3.5}
For a biconnected outerplane graph $G$ with two triangles, any same-color assignment to two independent vertices admits an extension to a 3-coloring of
$G$.
\end{theorem}
\begin{proof}
	Let $G$ be a biconnected outerplane graph with only two triangles and let $u$ and $v$ be the two nonadjacent vertices which have the same color, by Observation \ref{observation2} we add some chords to get a new graph $H$ and ensure that $u$ and $v$ are still nonadjacent vertices in $H$. Because $G$ is a subgraph of $H$, we only need to consider the coloring of $H$.  We will divide the graph $H$ into six cases based on the different types of triangles in the graph, and then discuss the following four cases regarding the positions of vertices $u$ and $v$:\\
	1. $u$ and $v$ are located in two different triangles.\\
	2. $u$ is in a triangle and $v$ is in a 4-face (or 5-face).\\
	3. $u$ and $v$ are in different 4-faces or 5-faces.\\
	4. $u$ and $v$ are in the same 4-face or 5-face.
	\setcounter{case}{0}
	\begin{case}\label{case1}
		The graph $H$ contains two marginal triangles.
	\end{case}
	When the vertices $u$ and $v$ are in two marginal triangles respectively, and they are the 2-degree vertices of the triangles. If a ``cake'' structure appears in $H$, we can have a three coloring on it, and then according to the conclusions of Theorem \ref{Grum and Ak} (ii), we obtain a three coloring for the remaining part of the graph, whether it consists of four or five face within the ``cake'' structure (Figure \ref{figure 9}). If there is no cake structure, due to every face of $H$ has length no more than 5, we can give homomorphisms from every face of $H$ into $K_3$. The three vertices of $K_3$ are denoted by $x$, $y$ and $z$ and we color them by $X$, $Y$, $Z$, respectively. Let the vertices of any 4-face (5-face) in the graph $H$ be $a$, $b$, $c$, $d$ ($a$, $b$, $c$, $d$, $e$) in the clockwise direction, and the face containing vertices $u$ and $v$ is colored by Definition \ref{hom mapping} and Operation \ref{hom operation}. If $\phi(u)=\phi(v)$, we complete the proof, otherwise we can use Algorithm \ref{algorithm1} to change the colors of the neighbors of $v$ by altering the mapping of $R_k$, thereby allowing $\phi(u)\notin\phi(N_{R_k}(v))$. After $u$ and $v$ are colored by the same color, the remaining uncolored faces can continue to follow the homomorphism mapping defined in Definition \ref{hom mapping}, the resulting entire graph $H$ can be 3-colored.
	\begin{remark}
		Because biconnected outerplane graphs do not contain subgraphs isomorphic to $K_4$ or $K_{2,3}$ minor, in any uncolored face, at most one edge has both endpoints colored by two colors. Since the length of any face in the graph $H$ is at most 5, every face is 3-colorable. Therefore, the uncolored faces can still be colored using the aforementioned homomorphism method, resulting in a proper 3-coloring of $H$.
	\end{remark}
	\begin{figure}[htbp]
		\centering
		\begin{minipage}[H]{0.6\linewidth}
			\begin{tikzpicture}[scale=0.9]
			\fill (-1.85,4) circle(2.5pt) ;
			\node at (-1.85,4) {};
			\fill (-2.5,3.25) circle(2.5pt) ;
			\node at (-2.5,3.25) {};
			\fill (-1,3.5) circle(2.5pt) ;
			\node at (-1,3.5) {};
			\fill (0,4) circle(2.5pt) ;
			\node at (0,4) {};
			\fill (0.65,3.3) circle(2.5pt) ;
			\node at (0.65,3.3) {};
			\fill(-1,2.2)circle(2.5pt) ;
			\node at (-1,2.2) {};
			\draw (-2.5,3.25)--(-1.85,4)--(-1,3.5)--(-2.5,3.25) node (v3) {}--(-1,2.2) node (v1) {}--(0.65,3.3);
			\draw (-1,3.5)--(0,4)--(0.65,3.3) node (v2) {}--(-1,3.5);
			
			\node at (-1.9,4.55) {$u$};
			\node at (0,4.55) {$v$};
			\fill (3.5,3.25) circle(2.5pt) ;
			\node at (3.5,3.25) {};
			\fill (5,3.5) circle(2.5pt) ;
			\node at (5,3.5) {};
			\fill (6,4) circle(2.5pt) ;
			\node at (6,4) {};
			\fill (6.65,3.3) circle(2.5pt) ;
			\node at (6.65,3.3) {};
			\fill(5,2.2)circle(2.5pt) ;
			\node at (5,2.2) {};
			\draw (5,3.5)--(3.5,3.25)--(5,2.2) node (v1) {}--(6.65,3.3);
			\draw (3.5,3.25) node (v11) {}--(5,3.5)--(6.65,3.3)--(6,4)--(5,3.5);
			\draw (5,3.5)--(6,4)--(6.65,3.3) node (v2) {}--(5,3.5);	
			\node at (6,4.55) {$v$};
			\node (v4) at (-2.95,2.5) {};\node (v4‘) at (3.05,2.5) {};
			\node (v5) at (-2.95,1.35) {};\node (v5’) at (3.05,1.35) {};
			\node (v6) at (-2.15,0.25) {};\node (v6‘) at (3.85,0.25) {};
			\node (v7) at (-0.7,-0.25) {};\node (v7’) at (5.3,-0.25) {};
			\node (v9) at (1.3,1.15) {};\node (v9‘) at (7.3,1.15) {};
			\node (v10) at (1.2,2.35) {};\node (v10’) at (7.2,2.35) {};
			\node (v8) at (0.6,0.2) {};\node (v8‘) at (6.6,0.2) {};
			\draw  plot[smooth, tension=.7] coordinates {(v3) (v4) (v5) (v6) (v7) (v8) (v9) (v10) (0.65,3.3)};
			\draw  plot[smooth, tension=.7] coordinates {(v11) (v4‘) (v5’) (v6‘) (v7’) (v8‘) (v9‘) (v10’) (v2)};
		\end{tikzpicture}
		\end{minipage}
		\caption{}
		\label{figure 9}
	\end{figure}

							For the case where the vertices $u$ and $v$ are in two marginal triangles, but at least one of them is not a 2-degree vertex, as well as the other three cases for the positions of $u$ and $v$, we use the same method to prove it. Let $w$ be the 2-degree vertex in another marginal triangle, and let $H^1=H-w$. Since $H^1$ is a biconnected outerplane graph containing exactly one triangle, Theorem \ref{Hong1} implies that $H^1$ is 3-colorable, hence $H$ is 3-colorable. Thus the conclusion holds.
							
							\begin{case}\label{case2}
								The graph $H$ contains a marginal triangle and a striped triangle.
							\end{case}
							When the vertices $u$ and $v$ are in a marginal triangle and a striped triangle, respectively. If there is a diamond structure in the graph, and the common neighbors of $u$ and $v$ are denoted as $u'$ and $v'$ (note that in this case, $u$ and $v$ can only be diamond vertices), let $H^2=H-u'v'$. Then, the diamond structure in graph $H$ becomes a 4-face in $H^2$. We color this 4-face by the three colors such that $u'$ and $v'$ receive different colors. By Theorem \ref{Grum and Ak}, $H^2$ is 3-colorable, then $H$ is also 3-colorable (Figure \ref{figure10}). Thus, the conclusion holds; If there is no diamond structure, and $u$ is a 2-degree vertex in the marginal triangle, and if $H$ contains a hamburger structure, we first color the hamburger structure with three colors such that the vertices $u$ and $v$ have the same color. Let $H^3=H-u$. By Theorem \ref{Grum and Ak}, $H^3$ is 3-colorable, then $H$ is also 3-colorable (Figure \ref{figure11}); if there is no hamburger structure, then the proof method is similar to the Case \ref{case1} where $u$ and $v$ are in the two triangles respectively and there is no cake structure. Using Definition \ref{hom mapping} and Operation \ref{hom operation} in conjunction with Algorithm \ref{algorithm1}, we ensure that $u$ and $v$ are colored by the same color, and $H$ is 3-colorable. If there is no diamond structure and neither $u$ nor $v$ is a 2-degree vertex, then let $w$ be the 2-degree vertex in the marginal triangle. Let $H^4=H-w$. By the conclusion of Theorem \ref{Hong1}, $H^4$ is 3-colorable, then $H$ is also 3-colorable.
							
							\begin{figure}[h!]
								\centering
								\begin{minipage}[H]{0.6\linewidth}
									\begin{tikzpicture}[scale=0.6]
										\draw  (0,0) node  {} circle (100pt);
										\draw  (12,0) node  {} circle (100pt);
										\fill (0,3.5) circle (2.5pt);	\fill (12,3.5) circle (2.5pt);	
										\node [above=4pt] at (0,3.5) {$\Large u$};\node[above=4pt] at (12,3.5) {$\Large u$};
										\fill (-2.55,2.4) circle (2.5pt);	\fill (9.45,2.4) circle (2.5pt);	
										\node[left=4pt] at (-2.55,2.4) {$\Large u'$};\node[left=4pt] at (9.45,2.4) {$\Large u'$};
										\fill (2.55,2.5) circle (2.5pt);	\fill (14.55,2.5) circle (2.5pt);	
										\node[right=4pt] at (2.55,2.5) {$\Large v'$};\node[right=4pt] at (14.55,2.5) {$\Large v'$};
										\fill (0,1.25) circle (2.5pt);	\fill (12,1.25) circle (2.5pt);	
										\node[below=4pt] at (0,1.25) {$\Large v$};\node[below=4pt] at (12,1.25) {$\Large v$};
										\draw  (-2.55,2.4)-- (0,1.25)--(2.55,2.5) ;
										\draw  (9.45,2.4)-- (12,1.25)--(14.55,2.5) ;
										\draw[thick,orange] (-2.55,2.4)--(2.55,2.5);
									\end{tikzpicture}
								\end{minipage}
								\caption{}
								\label{figure10}
							\end{figure}
							When $u$ is in a triangle and $v$ is in a $4$-face (or $5$-face). If $u$ is in a striped triangle, let $w$ be the 2-degree vertex in the marginal triangle, set $H^5=H-w$, by Theorem \ref{Hong1}, $H^5$ is 3-colorable, thus $H$ is also 3-colorable; if $u$ is in a marginal triangle and it is not a 2-degree vertex, then let $w$ be the 2-degree vertex of the marginal triangle, set the graph $H^6=H-w$, by Theorem \ref{Hong1}, the graph $H^6$ is 3-colorable, thus the graph $H$ is also 3-colorable; if $d(u)=2$ and there exists a diamond structure in the graph, let $w'$ be another diamond vertex and color it with the same color as $v$, set the graph $H^7=H-u$, again by Theorem \ref{Hong1}, the graph $H^7$ is 3-colorable, thus the graph $H$ is also 3-colorable; if there is no diamond structure in $H$, which can be 3-colored when $u$ and $v$ are precolored by the same color through Definition \ref{hom mapping} and Operation \ref{hom operation} in conjunction with Algorithm \ref{algorithm1}.
							\begin{figure}[h!]
								\centering
								\begin{minipage}[H]{0.65\linewidth}
									\begin{tikzpicture}[scale=0.6]
										\draw  (0,0) node  {} circle (100pt);
										\fill (0,3.5) circle (2.5pt);	
										\node [above=4pt] at (0,3.5) {$\Large u$};
										\node at (0,-3.5) {};\node (v4) at (12,-3.5) {};
										\fill (-2.55,2.4) circle (2.5pt);	\fill (9.45,2.4) circle (2.5pt);	
										\node[left=4pt] at (-2.55,2.4) {};\node[left=4pt] (v1) at (9.45,2.4) {};
										\fill (2.45,2.5) circle (2.5pt);	\fill (14.55,2.5) circle (2.5pt);	
										\node[right=4pt] at (2.45,2.5) {};\node[right=4pt] (v7) at (14.55,2.5) {};
										\node at (-2.45,-2.5) {};\node (v3) at (9.55,-2.5) {};
										\node at (2.5,-2.5) {}; \node (v5) at (14.5,-2.5) {};
										\fill (3.5,0) circle (2.5pt);	\fill (15.5,0) circle (2.5pt);	
										\node at (-3.5,0) {};\node (v2) at (8.5,0) {};
										\node at (3.5,0) {};\node (v6) at (15.5,0) {};
										\draw  plot[smooth, tension=.7] coordinates {(9.45,2.4) (v2) (v3) (v4) (v5) (v6) (14.55,2.5)};
										\draw (-3.35,1)--(3.5,0)--(-2.55,2.4)--(2.45,2.5);
										\draw(8.7,1)--(15.5,0)--(9.45,2.4)--(14.55,2.5);
										\fill (-3.35,1) circle (2.5pt);	\fill (8.7,1) circle (2.5pt);	
										\node at (-3.35,1) {};\node at (8.7,1) {};
										\fill (3.2,1.4) circle (2.5pt);	\fill (15.1,1.4) circle (2.5pt);	
										\node at (3.2,1.4) {};\node at (15.1,1.4) {};
									\end{tikzpicture}
								\end{minipage}
								\caption{}
								\label{figure11}
							\end{figure}
							
							When $u$ and $v$ are in different 4-faces (5-faces) or in the same 4-faces (5-faces), we can similarly denote $w$ as the 2-degree vertex in the marginal triangle, set $H^8=H-w$, by Theorem \ref{Hong1}, the graph $H^8$ is 3-colorable, thereby the graph $H$ is also 3-colorable.
							\begin{case}
								The graph $H$ contains a marginal and an internal triangle.
							\end{case}
							When the vertices $u$ and $v$ are in a marginal and an internal triangle, respectively. If there is a diamond structure in the graph, and the common neighbors of $u$ and $v$ are denoted as $u'$ and $v'$ (note that in this case, $u$ and $v$ can only be diamond vertices), let $H^2=H-u'v'$. Then, the diamond structure in the graph $H$ becomes a 4-face in $H^2$. We color this 4-face with three colors such that $u'$ and $v'$ receive different colors. By Theorem \ref{Grum and Ak}, $H^2$ is 3-colorable, then $H$ is also 3-colorable. Thus, the conclusion holds; If there is no diamond structure, and $u$ is a 2-degree vertex in the marginal triangle, and moreover the graph contains a hamburger structure, we first color the hamburger structure with the three colors such that the vertices $u$ and $v$ have the same color. Let $H^3=H-u$. By Theorem \ref{Grum and Ak}, $H^3$ is 3-colorable, then $H$ is also 3-colorable; if there is no hamburger structure, then the proof method is similar to that of Case \ref{case1} where $u$ and $v$ are in the two triangles respectively and there is no cake structure. Using Definition \ref{hom mapping} and Operation \ref{hom operation} in conjunction with Algorithm \ref{algorithm1}, we ensure that $u$ and $v$ have the same color, and $H$ is 3-colorable. If there is no diamond structure and neither $u$ nor $v$ is a 2-degree vertex, then let $w$ be the 2-degree vertex in the marginal triangle. Let $H^4=H-w$. By the conclusion of Theorem \ref{Hong1}, $H^4$ is 3-colorable, then $H$ is also 3-colorable.
							
							When $u$ is in a triangle and $v$ is in a 4-face (or 5-face). If the vertex $u$ is in an internal triangle, let $w$ be the 2-degree vertex in the marginal triangle, set $H^5=H-w$, by Theorem \ref{Hong1}, $H^5$ is 3-colorable, thus $H$ is also 3-colorable; if the vertex $u$ is in a marginal triangle and it is not a 2-degree vertex, then let $w$ be the 2-degree vertex of the marginal triangle, set $H^6=H-w$, by Theorem \ref{Hong1}, the graph $H^6$ is 3-colorable, thus the graph $H$ is also 3-colorable; if its degree is 2 and there exists a diamond structure in the graph, let $w'$ be another diamond vertex and color it with the same color as $v$, set $H^7=H-u$, by Theorem \ref{Hong1}, the graph $H^7$ is 3-colorable, thus the graph $H$ is also 3-colorable; if there is no diamond structure in the graph, using Definition \ref{hom mapping} and Operation \ref{hom operation} in conjunction with Algorithm \ref{algorithm1}, we ensure that $u$ and $v$ are colored by the same color, and $H$ is 3-colorable.
							
							When $u$ and $v$ are in different 4-faces (5-faces) or in the same 4-faces (5-faces), we can similarly denote $w$ as the 2-degree vertex in the marginal triangle, set $H^8=H-w$, by Theorem \ref{Hong1}, the graph $H^8$ is 3-colorable, thereby the graph $H$ is also 3-colorable.
							\begin{case}
								The graph $H$ contains two striped triangles.
							\end{case}
							When the vertices $u$ and $v$ are in two striped triangles, respectively.
							If there is a diamond structure in the graph, and the common neighbors of $u$ and $v$ are denoted as $u'$ and $v'$ (note that in this case, $u$ and $v$ can only be the diamond vertices), let $H^2=H-u'v'$. Then, the diamond structure in the graph $H$ becomes a 4-face in $H^2$. We color this 4-face with three colors such that $u'$ and $v'$ receive different colors. By Theorem \ref{Grum and Ak}, $H^2$ is 3-colorable, then $H$ is also 3-colorable; if there is no diamond structure in the graph as well as the other three situations, using Definition \ref{hom mapping} and Operation \ref{hom operation} in conjunction with Algorithm \ref{algorithm1}, we can ensure that $u$ and $v$ are colored by the same color, and $H$ is 3-colorable.
							\begin{case}
								The graph $H$ contains a striped and an internal triangle.
							\end{case}
							When the vertices $u$ and $v$ are located in a striped and an internal triangle, respectively. If there exists a diamond structure in the graph, and the common neighbors of $u$ and $v$ are denoted by $u'$ and $v'$ (it should be noted that in this case, $u$ and $v$ can only be the vertices of the diamond), let $H^2 = H - u'v'$. Subsequently, the diamond structure in the graph $H$ transforms into a 4-face in $H^2$. This 4-face is colored by the three colors such that $u'$ and $v'$ are assigned different colors. According to Theorem \ref{Grum and Ak}, $H^2$ is 3-colorable, then $H$ is also 3-colorable; in the cases where the graph does not contain a diamond structure or in the other three specific situations, by employing Definition \ref{hom mapping} and Operation \ref{hom operation} together with Algorithm \ref{algorithm1}, we can ensure that $u$ and $v$ are colored by the same color, and $H$ is 3-colorable.
							\begin{case}
								The graph $H$ contains two internal triangles.
							\end{case}
							When the vertices $u$ and $v$ are situated in two internal triangles, respectively. If there exists a diamond structure in the graph, and the mutual neighbors of $u$ and $v$ are denoted as $u'$ and $v'$ (it is important to mention that in this case, $u$ and $v$ can solely be the vertices of the diamond), let $H^2 = H - u'v'$. As the diamond configuration in the graph $H$ gets converted into a 4-face in $H^2$. The 4-face is colored by three colors, ensuring that $u'$ and $v'$ are assigned different colors. By Theorem \ref{Grum and Ak}, $H^2$ is 3-colorable and $H$ is also 3-colorable; in the cases where the graph lacks a diamond structure or in the other  three specific situations, by making use of Definition \ref{hom mapping} and Operation \ref{hom operation} along with Algorithm \ref{algorithm1}, we can guarantee that $u$ and $v$ are colored by the same way, and $H$ is 3-colorable.
						\end{proof}
						
						The combination of Theorems \ref{3.4} and \ref{3.5} resolves Theorem \ref{prove2}.
						
						 1-connected outerplane graphs is the subgraph of some 2-connected outerplane graphs by adding some edges from it. Thus, from the above argumentation process, the  Corollaries 1.7 and 1.8 hold.
						\section{Conclusion and Remarks}
						
						In this paper,  We have proposed the influence of diamond structures on the extendability problem in planar graphs, and combined with Gr\"{o}tzsch Theorem, we have characterized the structure of the 3-coloring extendability problem in outerplanar graphs through algorithm and Homomorphism mappings. The  {\bf Problems} 4.2, 4.3 and 4.4 at the end of the article~\cite{LaLuSt} are widely open.


						
					\end{document}